\documentclass[10pts]{amsart}

\usepackage{amssymb}
\usepackage{graphicx}
\usepackage{amsfonts}
\usepackage{amsmath}
\usepackage{amsthm}
\usepackage{fancyhdr}
\usepackage{indentfirst}
\usepackage{hyperref}
\usepackage{comment}
\usepackage{color}
\usepackage{subcaption}
\usepackage{epsfig}
\usepackage{pst-grad} % For gradients
\usepackage{pst-plot} % For axes
\usepackage[space]{grffile} % For spaces in paths
\usepackage{etoolbox} % For spaces in paths
\usepackage{mathrsfs} % For more math fonts
\newtheorem{theorem}{Theorem}[section]
\newtheorem{lemma}[theorem]{Lemma}

\theoremstyle{definition}
\newtheorem{definition}[theorem]{Definition}
\newtheorem{example}[theorem]{Example}
\newtheorem{corollary}[theorem]{Corollary}

\newtheorem{proposition}[theorem]{Proposition}

\theoremstyle{remark}
\newtheorem{remark}[theorem]{Remark}

\numberwithin{equation}{section}
\newcommand{\ep}{\epsilon}

\newcommand{\Om}{\Omega}

\newcommand{\R}{\mathbb{R}}

\newcommand{\Z}{\mathbb{Z}}

\newcommand{\lan}{\langle}
\newcommand{\ran}{\rangle}

\newcommand{\DN}{\Delta^{\frac{1}{2}}}

\newcommand{\Length}{\operatorname{Length}}

\begin{document}

\title[Chord Shortening Flow]{Chord Shortening Flow and a Theorem of Lusternik and Schnirelmann}
\date{\today}

\author[Martin Li]{Martin Man-chun Li}
\address{Department of Mathematics, The Chinese University of Hong Kong, Shatin, N.T., Hong Kong}
\email{martinli@math.cuhk.edu.hk}

\begin{abstract}
We introduce a new geometric flow called the chord shortening flow which is the negative gradient flow for the length functional on the space of chords with end points lying on a fixed submanifold in Euclidean space. As an application, we give a simplified proof of a classical theorem of Lusternik and Schnirelmann (and a generalization by Riede and Hayashi) on the existence of multiple orthogonal geodesic chords. For a compact convex planar domain, we show that any convex chord which is not orthogonal to the boundary would shrink to a point in finite time under the flow.
\end{abstract}

\maketitle
%\tableofcontents

%%%%%%%%%%%%%%%%%
% Section  Introduction
\section{Introduction}
\label{S:intro}
%%%%%%%%%%%%%%%%%

The existence of closed geodesics in a Riemannian manifold is one of the most fundamental questions in geometry that has been studied extensively since the time of Poincar\'{e} \cite{Poincare04}. The critical point theories developed by Morse and Lusternik-Schnirelmann have played an essential role in this problem in the early 20th century (see \cite{Klingenberg} for a detailed exposition up to 1978). Although there does not exist closed geodesics in $\R^n$, it is natural to look for geodesics contained in a bounded domain $\Omega \subset \R^n$ which meets $\partial \Omega$ orthogonally at its end points. These are called \emph{orthogonal geodesic chords} (see Definition \ref{D:OGC} for a precise definition). In \cite{Lusternik-Schnirelmann}, Lusternik and Schnirelmann proved the following celebrated result:

\begin{theorem}[Lusternik-Schnirelmann]
\label{T:LS}
Any bounded domain in $\mathbb{R}^n$ with smooth convex boundary contains at least $n$ distinct orthogonal geodesic chords.
\end{theorem}

Kuiper \cite{Kuiper64} showed that the same conclusion holds if the boundary is only $C^{1,1}$. For our convenience, we will assume that all the submanifolds and maps are $C^\infty$. Recall that the boundary of a domain $\Omega \subset \R^n$ is said to be (locally) \emph{convex} if the second fundamental form $A$ of $\partial \Omega$ with respect to the unit normal $\nu$ (pointing into $\Omega$) is positive semi-definite, i.e. for all $p \in \partial \Omega$, $u \in T_p \partial \Omega$, we have
\begin{equation}
\label{E:convex}
A(u,u):=\langle D_u u ,\nu \rangle  \geq 0,
\end{equation}
where $D$ is the standard flat connection in $\R^n$. Notice that Theorem \ref{T:LS} gives an optimal lower bound as seen in the example of the convex region bounded by the ellipsoid given by
\[ \Omega:=\left\{(x_1,\cdots,x_n) \in \R^n \; : \; \sum_{i=1}^n \frac{x_i^2}{a_i^2} \leq 1 \right\} \]
where $a_1,\cdots,a_n$ are distinct positive real numbers. 

In \cite{Bos63}, Bos generalized Lusternik-Schnirelmann's result to the setting of Riemannian (or even Finsler) manifolds.

\begin{theorem}[Bos]
\label{T:Bos}
A compact Riemannian manifold $(M^n,g)$ which is homeomorphic to the closed unit ball in $\R^n$ with locally convex boundary contains at least $n$ orthogonal geodesic chords.
\end{theorem}

Moreover, he showed that the convexity assumption cannot be dropped even in $\R^2$ (see Figure \ref{F:Bos}). 

\begin{figure}[h]
\centering
\includegraphics[height=4cm]{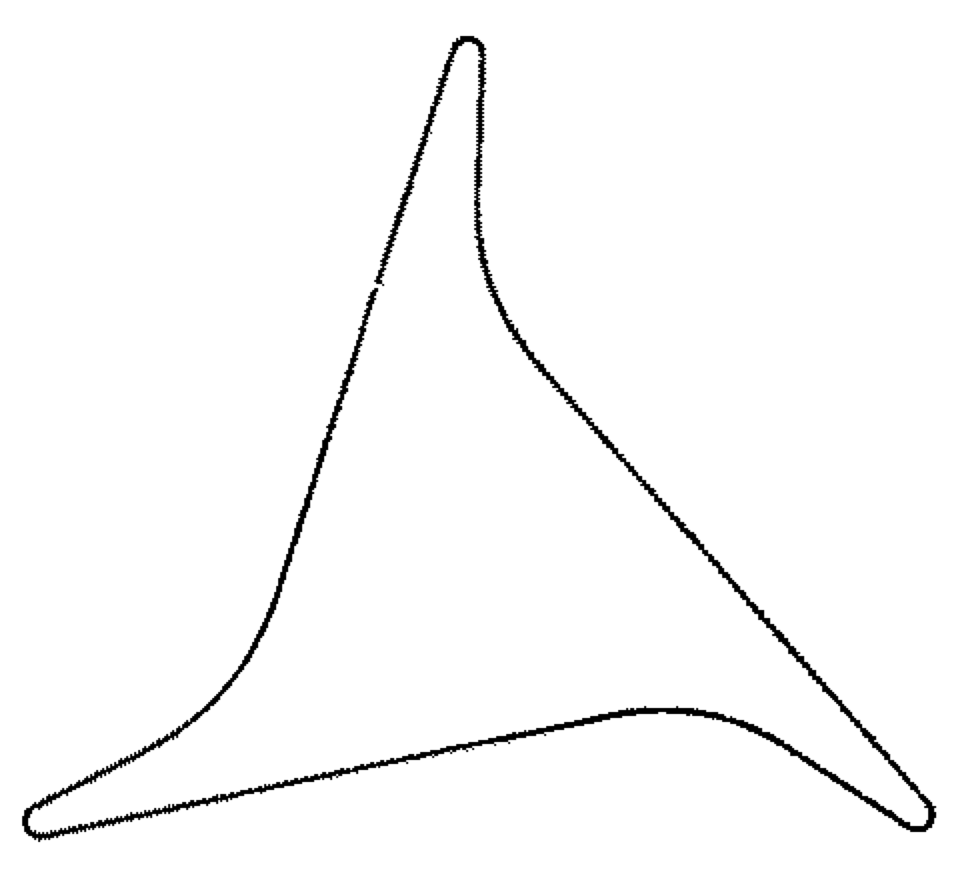}
\caption{Bos' example of a non-convex domain $\Omega$ in $\R^2$ which does not have any orthogonal geodesic chord contained in $\Omega$.}
\label{F:Bos}
\end{figure}

Nonetheless, one can still ask for the existence of orthogonal geodesic chords, by allowing them to go \emph{outside} the domain. This problem was first introduced by Riede \cite{Riede68}, where he studied the variational calculus of the space $\Gamma$ consisting of piecewise smooth curves in a complete Riemannian manifold $(M^n,g)$ with end points lying on a compact submanifold $\Sigma^k \subset M$. In particular, he estimated the minimum number of critical points, which are orthogonal geodesic chords, in terms of certain topological invariant called the ``cup-length'' of the equivariant cohomology of $\Gamma$ with respect to the $\Z_2$-action reversing the orientation of a curve. In \cite{Hayashi82}, Hayashi computed the cup-length when $\Sigma$ is a compact submanifold in $\R^n$ and hence proved the following result.

\begin{theorem}[Riede-Hayashi]
\label{T:double-normals}
Any $k$-dimensional compact submanifold $\Sigma$ in $\R^n$ admits at least $k+1$ orthogonal geodesic chords.
\end{theorem}

Note that Theorem \ref{T:double-normals} generalizes Theorem \ref{T:LS} by taking $\Sigma$ to be the boundary of a bounded convex domain. However, we emphasize that if $\Sigma=\partial \Omega$ is the boundary of a non-convex domain $\Omega \subset \R^n$, then the orthogonal geodesic chords obtained in Theorem \ref{T:double-normals} are not necessarily contained in $\Omega$ (recall Figure \ref{F:Bos}).

The original proof of Theorem \ref{T:LS}, \ref{T:Bos} and \ref{T:double-normals} all used a discrete curve shortening process similar to the one introduced by Birkhoff \cite{Birkhoff17} in the study of existence of closed geodesics in Riemannian manifolds. A description of the process can be found in \cite{Gluck-Ziller83} (see also a modified version in \cite{Zhou16}). The curve shortening process, denoted by $\Psi$, take a piecewise smooth curve $c:[0,1] \to M$ with end points lying on $\Sigma$ to a piecewise geodesic curve $\Psi(c):[0,1] \to M$ which meets $\Sigma$ orthogonally at its end points. The most important properties of $\Psi$ are summarized below:
\begin{itemize}
\item[(1)] $\Length(\Psi(c)) \leq \Length(c)$ and equality holds if and only if $c$ is an orthogonal geodesic chord, in which case $\Psi(c)=c$.
\item[(2)] $\Psi(c)$ depends continuously on $c$, with respect to the $C^0$ topology.
\item[(3)] $c$ and $\Psi(c)$ are homotopic in $M$ relative to $\Sigma$, i.e. there exists a continuous family $c_t:[0,1] \to M$, $t \in [0,1]$, with end points on $\Sigma$ such that $c_0=c$ and $c_1=\Psi(c)$. Moreover, the family $c_t$ depends continuously on $c$.
\end{itemize}
The curve shortening process $\Psi$ involves subdividing the curves and connecting points on the curve by minimizing geodesic segments (additional care has to be taken at the end points). The construction depends on some fixed parameter (which depends on the geometry of $M$, $\Sigma$ and $\Length(c)$). However, it can be shown that for curves with uniformly bounded length, the parameters can be chosen uniformly to make (1) - (3) above hold. In fact (1) and (3) follows easily from the constructions, but (2) requires some convexity estimates (see \cite[Lemma 3.2]{Zhou16}). Using (1) - (3), it is not difficult to see that the sequence $\{\Psi^i(c)\}_{i=1}^\infty$ either converges to a point on $\Sigma$ or has a subsequence converging to an orthogonal geodesic chord. Theorem \ref{T:LS}, \ref{T:Bos} and \ref{T:double-normals} then follows from the abstract Lusternik-Schnirelmann theory applied to families of curves with end points on $\Sigma$ which represent a non-trivial homology class relative to point curves on $\Sigma$. Interested readers can refer to \cite{Gluck-Ziller83} or \cite{Giannoni-Majer97} for more details (for Theorem \ref{T:LS} there is a more elementary proof - see \cite{Kuiper64} for example).

In this paper, we introduce a new curve shortening process called the \emph{chord shortening flow} (see Definition \ref{D:CSF}), which evolves a geodesic chord according to the ``contact angle'' that the chord makes with $\Sigma$ at its end points. It is the negative gradient flow for the length functional on the space of chords. We study the fundamental properties including the short-time existence and uniqueness and long-time convergence of the flow when the ambient space is $\R^n$. Note that the flow still makes sense in certain Riemannian manifolds but for simplicity we postpone the details to another forthcoming paper. The chord shortening flow, as a negative gradient flow, clearly satisfies all the properties (1) - (3) above; hence provide the most natural curve shortening process required in the proof of Theorem \ref{T:LS} and \ref{T:double-normals} (but not Theorem \ref{T:Bos} in its full generality). 

\begin{remark}
We would like to mention that Lusternik and Schnirelmann used the same ideas to prove the \emph{Theorem of Three Geodesics} which asserts that any Riemannian sphere $(S^2,g)$ contains at least three \emph{geometrically distinct} closed embedded geodesics. Unfortunately, the original proof by Lusternik-Schnirelmann \cite{Lusternik-Schnirelmann} contains a serious gap and various attempts have been made to fix it (see \cite{Taimanov92}). The fundamental issue there is \emph{multiplicity}, that one of the geodesics obtained may just be a multiple cover of another geodesic. It is extremely technical (and many false proofs were given) to rule out this situation by modifying the method of Lusternik-Schnirelmann. In \cite{Grayson89}, Grayson gave a rigorous proof of the Theorem of Three Geodesics by a careful analysis of the curve shortening flow on Riemannian surfaces. He proved that under the curve shortening flow, any embedded curve remains embedded and would either converge to a point in finite time or an embedded closed geodesic as time goes to infinity. As a curve which is initially embedded stays embedded throughout the flow, this prevents the multiplicity problem encountered by Lusternik-Schnirelmann's approach using a discrete curve shortening process of Birkhoff \cite{Birkhoff17}. On the other hand, the situation in Theorem \ref{T:LS} and \ref{T:double-normals} are simpler as multiplicity cannot occur (see \cite[Remark 3.2]{Giannoni-Majer97}).
\end{remark}

We show that the convergence behavior for the chord shortening flow is similar to that for the curve shortening flow on a closed Riemannian surface \cite{Grayson89}. In particular, we prove that under the chord shortening flow, any chord would either converge to a point in finite time or to an orthogonal geodesic chord as time goes to infinity. Unlike \cite{Grayson89}, this dichotomy holds in any dimension and codimension, in contrast with the curve shortening flow where an embedded curve may develop self-intersections or singularities after some time when codimension is greater than one \cite{Altschuler91}. In the special case that $\Sigma=\partial \Omega$ where $\Omega \subset \R^2$ is a compact convex planar domain, we give a sufficient condition for an initial chord to converge to a point in finite time. In fact, any ``\emph{convex}'' chord in $\Omega$ which is not an orthogonal geodesic chord would converge to a point on $\partial \Omega$ in finite time. This can be compared to the famous result of Huisken \cite{Huisken84} which asserts that any compact embedded convex hypersurface in $\R^n$ converges to a point in finite time under the mean curvature flow.
%Roughly speaking, if the initial chord is ``convex enough'', then it will converge to a point under the chord shortening flow. Our result is analogous to Huisken's theorem \cite{Huisken86} which says that any hypersurface in a Riemannian manifold which is ``convex enough'' would shrink to a point in finite time under the mean curvature flow. As in \cite{Huisken86}, we derive some consequences of such a convergence result. 

The chord shortening flow is also of independenet interest from the analytic point of view. Since any chord in $\R^n$ is determined uniquely by its end points, we can regard the chord shortening flow as an evolution equation for the two end points lying on $\Sigma$. As a result, the flow is a \emph{non-local} evolution of a pair of points on $\Sigma$ as it depends on the chord joining them. In fact, the chord shortening flow can be regarded as the heat equation for the half-Laplacian (or the \emph{Dirichlet-to-Neumann map}).
%and the geometric quantities also satisfy similar the evolution equations. 

The organization of this paper is as follows. In Section 2, we introduce the chord shortening flow, give a few examples, and prove the short time existence and uniqueness of the flow. In Section 3, we derive the evolution equations for some geometric quantities under the chord shortening flow. In Section 4, we prove the long time existence to the flow provided that it does not shrink the chord to a point in finite time. In Section 5, we prove that an initial convex chord inside a compact convex domain in $\R^2$ would shrink to a point in finite time under the flow, provided that the initial chord is not an orthogonal geodesic chord. 

\vspace{.3cm}

\noindent \textbf{Acknowledgement.} The author would like to express his gratitude to Prof. Richard Schoen for his interest in this work. He also want to thank Mario Micallef and Luen-Fai Tam for helpful comments and discussions. These work are partially supported by a research grant from the Research Grants Council of the Hong Kong Special Administrative Region, China [Project No.: CUHK 14323516] and CUHK Direct Grant [Project Code: 3132705].

\vspace{.3cm}

\noindent \textit{Notations.} Throughout this paper, we will denote $I:=[0,1]$ with $\partial I=\{0,1\}$. The Euclidean space $\R^n$ is always equipped with the standard inner product $\langle \cdot, \cdot \rangle$ and norm $|\cdot |$. For any subset $S \subset \R^n$, we use $d(\cdot, S)$ to denote the distance function from $S$. 
\section{Chord Shortening Flow}
\label{S:CSF}
%%%%%%%%%%%%%%%%%

In this section, we introduce a new geometric flow called \emph{chord shortening flow}. This flow has some similarities with the classical curve shortening flow. 
%yet it is \emph{non-local}. 
The main result in this section is the short-time existence and uniqueness theorem for the chord shortening flow (Theorem \ref{T:short-time-existence}). We also study some basic examples of such a flow. 
%Note that we define the chord shortening flow in a more general setting than needed to prove Theorem \ref{T:LS}.

Let $\Sigma$ be a $k$-dimensional smooth submanifold \footnote{In fact all the following discussions make sense for \emph{immersed} submanifolds. However, for simplicity, we will assume that all submanifolds are \emph{embedded}.} in $\mathbb{R}^n$. Note that $\Sigma$ can be disconnected in general. For any two points $p,q \in \Sigma$, we can consider the extrinsic chord distance between them in $\mathbb{R}^n$. 
%Note that this is sometimes called the \emph{boundary distance function} (see for example \cite{Pestov-Uhlmann05}).

%Let $\Omega \subset \mathbb{R}^2$ be a smooth connected closed subset whose boundary $\partial \Omega := \overline{\Omega} \setminus \Omega$ is an embedded $C^\infty$ curve (which may be disconnected). In this paper we will denote $\langle \cdot, \cdot \rangle$ to be the standard inner product in $\mathbb{R}^2$, and $I:=[0,1]$ with $\partial I=\{0,1\}$. For any two points $p,q \in \partial \Omega$, we can consider the extrinsic chord distance between them in $\mathbb{R}^2$. Note that this is sometimes called the \emph{boundary distance function} (see for example \cite{Pestov-Uhlmann05}).

\begin{definition}
\label{D:chord-distance}
The \emph{chord distance function} $d: \Sigma \times \Sigma \to \mathbb{R}_{\geq 0}$ is defined to be 
\[ d(p,q):=\textrm{dist}_{\mathbb{R}^n} (p,q)=|p-q|. \]
%where $\textrm{dist}_{\mathbb{R}^n} (p,q):=| p-q|$ is the standard Euclidean distance between $p$ and $q$.
\end{definition}

Since any two distinct points in $\mathbb{R}^n$ are connected by a unique line segment realizing their distance, the chord distance function $d$ is smooth away from the diagonal $\{(p,p) \in \Sigma \times \Sigma \, : \, p \in \Sigma\}$. 

\begin{definition}
\label{D:Cpq}
For any two distinct points $p,q$ on $\Sigma$, we will use $C_{p,q}$ to denote the unique oriented chord from $p$ to $q$. The outward unit conormal, denoted by $\eta$, is the unit vector at $\partial C_{p,q}$ tangent to $C_{p,q}$ pointing out of $C_{p,q}$. Note that $\eta(p)=-\eta(q)$. (see Figure \ref{F:Cpq})
\end{definition}

\begin{figure}[h]
\centering
\includegraphics[height=5cm]{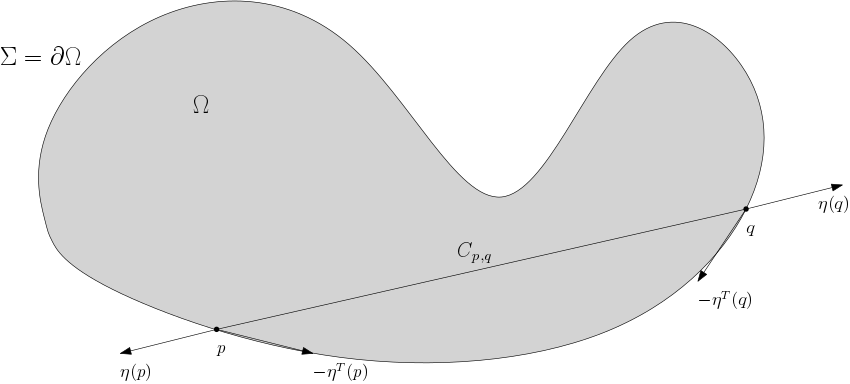}
\caption{A chord $C_{p,q}$ joining $p$ and $q$, the outward unit conormals $\eta$ and their (negative) tangential components along $\Sigma=\partial \Om$}
\label{F:Cpq}
\end{figure}

%\begin{remark}
%Note that $C_{p,q}$ may not be contained in $\Omega$ (unless $\Omega$ is a convex domain). Moreover, we have $\eta(q)=-\eta(p)$.
%\end{remark}

Let $C(t)=C_{p_t,q_t}$ be a smooth family of chords with distinct end points $p_t,q_t \in \Sigma$. If $\ell(t)=d(p_t,q_t)$ is the length of the chord $C(t)$, the first variation formula for arc length (see for example \cite[(1.5)]{Cheeger-Ebin}) implies that
\begin{equation}
\label{E:1st-var-a} 
\frac{d \ell}{dt} =\langle \frac{dp_t}{dt}, \eta(p_t) \rangle + \langle \frac{dq_t}{dt}, \eta(q_t) \rangle. 
%&=& \langle \frac{dp_t}{dt}, \pi_{p_t} \eta(p_t) \rangle + \langle \frac{dq_t}{dt}, \pi_{q_t} \eta(q_t) \rangle, \nonumber
\end{equation}
Note that the interior integral term vanishes as $C(t)$ is a geodesic for every $t$. Since $p_t$ and $q_t$ lies on $\Sigma$ for all $t$, both $dp_t/dt$ and $dq_t/dt$ are tangential to $\Sigma$. Therefore, we can express (\ref{E:1st-var-a}) as
\begin{equation}
\label{E:1st-var-b}
\frac{d \ell}{dt} = \langle \frac{dp_t}{dt}, \eta^T(p_t) \rangle + \langle \frac{dq_t}{dt}, \eta^T(q_t) \rangle
\end{equation}
where $(\cdot)^T$ denotes the tangential component of a vector relative to $\Sigma$. More precisely, if $\pi_x:\R^n \to T_x\Sigma$ is the orthogonal projection onto the tangent space $T_x \Sigma \subset \R^n$, then $v^T=\pi_x(v)$ for any vector $v \in T_x \R^n \cong \R^n$. 

It is natural to consider the (negative) gradient flow to the chord length functional, which leads to the following definition.
%on the space of distinct points $p,q$ on $\partial \Omega$ given by the following system of ODEs (see Figure \ref{F:Cpq}):
%\begin{eqnarray}
%\label{E:CSF'}
%\frac{dp_t}{dt} &=& -\eta^T(p_t), \\
%\frac{dq_t}{dt} &=& -\eta^T(q_t). \nonumber
%\end{eqnarray}

%Instead of flowing the pair of points $p$ and $q$, we can equivalently consider the flow of the chord $C_{p,q}$ joining them as follows.

\begin{definition}[Chord Shortening Flow]
\label{D:CSF}
A smooth family of curves 
\[ C(u,t):I \times [0,T) \to \mathbb{R}^n\]
is a solution to the \emph{chord shortening flow} (relative to $\Sigma$) if for all $t \in [0,T)$, we have
\begin{itemize}
\item[(a)] $p_t:=C(0,t)$ and $q_t:=C(1,t)$ lies on $\Sigma$,
\item[(b)] $C(t):=C(\cdot,t):I \to \mathbb{R}^n$ is a constant speed parametrization of $C_{p_t,q_t}$,
\item[(c)] 
\[ \frac{\partial C}{\partial t} (0,t) = - \eta^T(C(0,t)) \qquad \text{and} \qquad \frac{\partial C}{\partial t} (1,t) = - \eta^T(C(1,t)).\]
\end{itemize}
\end{definition}

Let us begin with some basic examples of the chord shortening flow as defined in Definition \ref{D:CSF}. 

\begin{example}
\label{ex:flat}
Let $\Sigma$ be an affine $k$-dimensional subspace in $\R^n$. The chord shortening flow with respect to $\Sigma$ will contract any initial chord $C(0)=C_{p,q}$ to a point in finite time. The end points would move towards each other with unit speed along the chord $C(0)$ until they meet at the mid-point of $C(0)$ at the ``blow-up'' time $T=d(p,q)/2$.
\end{example}

\begin{example}
\label{ex:annulus}
Let $\Sigma$ be a union of two disjoint circles in $\R^2$. We will see (from Theorem \ref{T:convergence}) that any chord joining two distinct connected components of $\Sigma$ would evolve under the chord shortening flow to a limit chord $C_\infty$ orthogonal to $\Sigma$ as $t \to \infty$. The same phenomenon holds for any $\Sigma \subset \R^n$ which is disconnected.
%Note that, however, the chord $C_\infty$ may not lie inside $\Omega$, see Figure \ref{F:annulus}.
\end{example}

\begin{figure}[h]
\centering
\includegraphics[height=5cm]{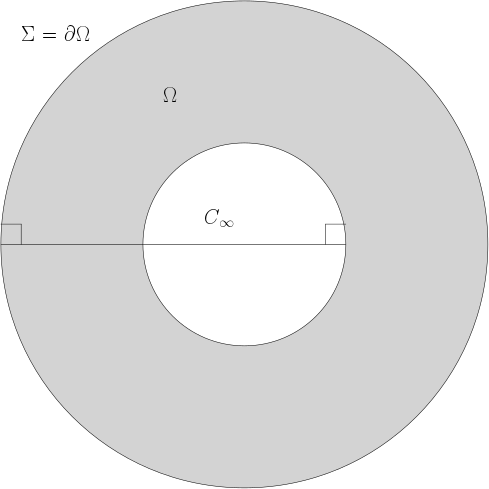}
\caption{A limit chord $C_\infty$ which meets $\partial \Omega$ orthogonally but not lying inside $\Omega$.}
\label{F:annulus}
\end{figure}

\begin{example}
Let $\Sigma$ be the ellipse $\{(x,y) \in \R^2 \, : \, x^2+4y^2=1\}$ in $\R^2$. By symmetry it is not difficult to see that for any initial chord passing through the origin (with the exception of the major axis), it would evolve under the chord shortening flow to the minor axis of the ellipse, which is a chord orthogonal to $\Sigma$ and contained inside the region enclosed by the ellipse. See Figure \ref{F:ellipse}. This example shows that the number of distinct orthogonal chords guaranteed by the Lyusternik-Schnirelmann Theorem is optimal. If we start with an initial chord that lies completely on one side of the major or minor axis, then the chord will shrink to a point in finite time (by Theorem \ref{T:contraction}).
\end{example}

\begin{figure}[h]
\centering
\includegraphics[height=5cm]{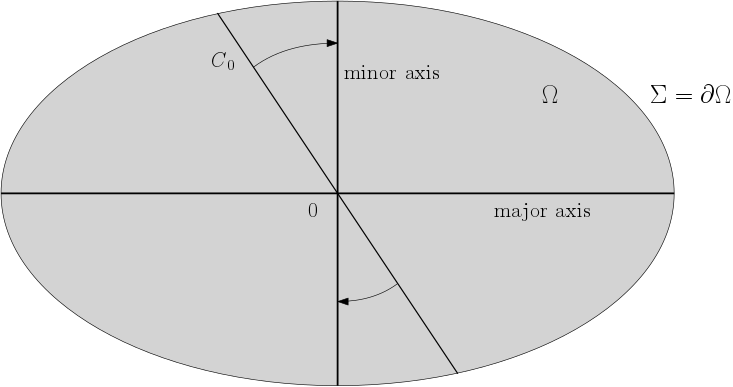}
\caption{Any initial chord $C_0$ through the origin (other than the major axis) would converge under the chord shortening flow to the minor axis.}
\label{F:ellipse}
\end{figure}

We end this section with a fundamental result on the short time existence and uniqueness for the chord shortening flow. 

\begin{proposition}[Short-time existence and uniqueness]
\label{T:short-time-existence}
For any initial chord $C_0:I \to \mathbb{R}^n$ with $C_0(\partial I) \subset \Sigma$, there exists an $\epsilon>0$ and a smooth solution $C(u,t):I \times [0,\epsilon) \to \mathbb{R}^n$ to the chord shortening flow relative to $\Sigma$ as in Definition \ref{D:CSF} with initial condition $C(\cdot,0)=C_0$. Moreover, the solution is unique.
\end{proposition}

\begin{proof}
Note that for any given $p \neq q \in \Sigma$, the outward unit conormal $\eta$ at the end points $p,q$ of the chord $C_{p,q}$ is given by
\[ \eta(p)=\frac{p-q}{|p-q|} =-\eta(q). \]
Therefore, Definition \ref{D:CSF} (c) is equivalent to the following system of nonlinear system of first order ODEs:
\begin{equation}
\label{E:CSF'}
\left\{ \begin{array}{l}
\frac{dp}{dt}=- \frac{\pi_p(p-q)}{|p-q|} \\
\frac{dq}{dt}=- \frac{\pi_q(q-p)}{|q-p|} 
\end{array} \right.
\end{equation}
where $\pi_x:\R^n \to \R^n$ is the orthogonal projection onto $T_x\Sigma$ (which depends smoothly on $x$). Since the right hand side of (\ref{E:CSF'}) is a Lipschitz function in $p$ and $q$ as long as $|p-q|$ is bounded away from $0$. Therefore, the existence and uniqueness to the initial value problem follows from the fundamental local existence and uniqueness theorem for first order ODE systems (see for example \cite[Theorem 2.1]{Taylor}). Hence, (\ref{E:CSF'}) is uniquely solvable on some interval $t \in [0,\epsilon)$ for any initial data $p(0)=p_0$ and $q(0)=q_0$ such that $p_0 \neq q_0 \in \Sigma$. Finally we get a solution $C(u,t):I \times [0,\epsilon) \to \R^n$ to the chord shortening flow by defining $C(\cdot,t):I \to \R^n$ to be the constant speed parametrization of the chord $C_{p_t,q_t}$.
\end{proof}

%%%%%%%%%%%%%%%%%
% Section   Evolution equations
\section{Evolution equations}
\label{S:evolution}
%%%%%%%%%%%%%%%%%

In this section, we derive the evolution of some geometric quantities under the chord shortening flow relative to any $k$-dimensional submanifold $\Sigma$ in $\R^n$. 

\begin{definition}
Let $C:I=[0,1] \to \R^n$ be a chord joining $p$ to $q$. For any (vector-valued) function $f:\partial I =\{ 0,1\} \to \R^m$, we define the \emph{$L^2$-norm} $\|f\|_{L^2}$ and the \emph{sum} $\overline{f}$ of $f$ to be\begin{equation}
\label{E:norm}
\|f\|_{L^2}:=( |f(0)|^2+|f(1)|^2)^{\frac{1}{2}} \qquad \text{and} \qquad \overline{f}:=f(0)+f(1).
\end{equation}
Also, we define the $\frac{1}{2}$-Laplacian of $f$ relative to the chord $C$ to be the vector-valued function $\DN f:\partial I=\{0,1\} \to \R^m$ defined by
\begin{equation}
\label{E:DN}
(\Delta^{\frac{1}{2}} f)(0)= \frac{f(0)-f(1)}{\ell} =- (\Delta^{\frac{1}{2}} f)(1),
\end{equation}
where $\ell=|p-q|$ is the length of the chord $C$.
\end{definition}

\begin{lemma}
\label{L:DN}
Given any $f:\partial I \to \R^m$, we have $\overline{\DN f}=0$ and $\overline{\lan f, \DN f \ran}=\frac{\ell}{2} \|\DN f\|^2_{L^2} \leq \frac{2}{\ell} \|f\|^2_{L^2}$.
\end{lemma}
\begin{proof}
It follows directly from (\ref{E:norm}) and (\ref{E:DN}).
\end{proof}

\begin{definition}
Let $C=C_{p,q}:I \to \R^n$ be a chord joining two distinct points $p,q$ on $\Sigma$. We define the tangential outward conormal $\eta^T:\partial I=\{0,1\} \to \R^n$ to be the tangential component (relative to $\Sigma$) of the outward unit conormal of $C$, i.e. (recall (\ref{E:1st-var-b}) and Definition \ref{D:Cpq})
\begin{equation}
\label{E:eta-T}
\eta^T(u)=\pi_{C(u)} \eta \qquad \text{for $u=0,1$}.
\end{equation}
\end{definition}

\begin{lemma}[Evolution of chord length]
\label{L:length-evolution}
Suppose $C(u,t):I \times [0,T) \to \R^n$ is a solution to the chord shortening flow relative to $\Sigma$ as in Definition \ref{D:CSF}. If we denote the length of the chord $C(t)$ at time $t$ by
\[ \ell(t):=d(C(0,t),C(1,t)),\]
then $\ell$ is a non-increasing function of $t$ and (recall (\ref{E:norm}) and (\ref{E:eta-T}))
\begin{equation}
\label{E:ell-evolution}
\frac{d \ell}{dt} = - \|\eta^T\|_{L^2}^2 \leq 0 .
\end{equation}
\end{lemma}

\begin{proof}
It follows directly from the first variation formula (\ref{E:1st-var-b}).
\end{proof}

\begin{theorem}
\label{T:eta-T-evolution}
Suppose $C(u,t):I \times [0,T) \to \R^n$ is a solution to the chord shortening flow relative to $\Sigma$ as in Definition \ref{D:CSF}. Then the tangential outward conormal $\eta^T$ of the chord $C(t)$ satisfies the following evolution equation:
\begin{eqnarray}
\label{E:eta-T-evolution}
\frac{\partial}{\partial t} \eta^T &=& - \DN \eta^T + \frac{1}{\ell} \|\eta^T\|^2_{L^2} \eta^T  - \sum_{i=1}^k \lan A(\eta^T,e_i),\eta^N \ran e_i  \\
&& \phantom{aaaaaaaaa} - \frac{1}{\ell}  (\overline{\eta^T}-\eta^T)^N - A(\eta^T,\eta^T), \nonumber
\end{eqnarray}
where $\{e_i\}_{i=1}^k$ is an orthonormal basis of $T\Sigma$ at the end points of $C(t)$. Here, $(\cdot)^N$ denotes the normal component of a vector relative to $\Sigma$ and $A:T\Sigma \times T\Sigma \to N\Sigma$ is the second fundamental form of $\Sigma$ defined by $A(u,v):=(D_u v)^N$.
\end{theorem}

\begin{proof}
Let $C(u,t):I \times [0,T) \to \mathbb{R}^n$ be a solution to the chord shortening flow relative to $\Sigma$. Since $C(t)=C( \cdot, t):I \to \mathbb{R}^n$ is a family of chords which are parametrized proportional to arc length, $\frac{\partial}{\partial t}$ is a Jacobi field (not necessarily normal) along each chord which can be explicitly expressed as
\[ \frac{\partial }{\partial t} = -(1-u) \,  \eta^T(0) - u \, \eta^T(1),\]
where $\eta$ is the outward unit conormal for $C(t)$. Since $[\frac{\partial}{\partial u}, \frac{\partial }{\partial t}]=0$, we have
\begin{equation}
\label{E:commutator}
D_{\frac{\partial}{\partial t}} \frac{\partial }{\partial u}=D_{\frac{\partial}{\partial u}} \frac{\partial }{\partial t} = \eta^T(0)-\eta^T(1). 
\end{equation}
Moreover, as $C(t)$ is parametrized with constant speed, we have $\| \frac{\partial}{\partial u}\|= \ell$, thus
\[ -\eta(0)= \frac{1}{\ell} \left.\frac{\partial }{\partial u} \right|_{u=0} \quad  \text{and} \quad \eta(1)=\frac{1}{\ell} \left. \frac{\partial }{\partial u} \right|_{u=1}.\]
Fix $u=0$. Let $p=C(0,t) \in \Sigma$ and $\{e_1,\cdots, e_k\}$ be an orthonormal basis of $T_p \Sigma$ such that $(D_{e_i} e_j(p))^T=0$ for $i,j=1,\cdots,k$. Therefore, we have
\begin{equation}
\label{E:A}
D_{\frac{\partial}{\partial t}} e_i = -A(\eta^T,e_i). 
\end{equation}
Using Lemma \ref{L:length-evolution}, (\ref{E:commutator}) and (\ref{E:A}), we have:
\begin{eqnarray*}
\frac{\partial \eta^T}{\partial t} &=& \frac{\partial}{\partial t} \left( -\frac{1}{\ell} \right) \sum_{i=1}^k \langle \frac{\partial }{\partial u}, e_i \rangle e_i- \frac{1}{\ell} \sum_{i=1}^k \frac{\partial }{\partial t} \left( \langle \frac{\partial }{\partial u},e_i \rangle  e_i \right)\\
&=& \frac{1}{\ell} \|\eta^T\|^2_{L^2} \eta^T - \frac{1}{\ell} \sum_{i=1}^k \left( \langle D_{\frac{\partial}{\partial u}} \frac{\partial}{\partial t}, e_i \rangle e_i+ \langle \frac{\partial}{\partial u}, D_{\frac{\partial}{\partial t}} e_i \rangle e_i  + \langle \frac{\partial }{\partial u},e_i \rangle D_{\frac{\partial}{\partial t}} e_i \right) \\
&=& \frac{1}{\ell} \|\eta^T\|^2_{L^2} \eta^T - \frac{\eta^T}{\ell} - A(\eta^T,\eta^T) - \frac{1}{\ell} \sum_{i=1}^k \left( \lan -\eta^T(1),e_i \ran e_i + \ell \langle \eta^N, A(\eta^T,e_i) \rangle e_i  \right) \\
&=&  - \DN \eta^T + \frac{1}{\ell} \|\eta^T\|^2_{L^2} \eta^T  - \sum_{i=1}^k \lan A(\eta^T,e_i),\eta^N \ran e_i - \frac{1}{\ell}  (\overline{\eta^T}-\eta^T)^N - A(\eta^T,\eta^T).
\end{eqnarray*}
A similar calculation yields (\ref{E:eta-T-evolution}) at $u=1$. This proves the proposition.
\end{proof}

\begin{corollary}
Under the same assumptions as in Theorem \ref{T:eta-T-evolution}, we have
\begin{equation}
\label{E:eta-norm-evolution}
\frac{1}{2} \frac{d}{dt} \|\eta^T\|^2_{L^2}= - \frac{\ell}{2}  \| \DN \eta^T \|^2_{L^2}  + \frac{1}{\ell} \|\eta^T\|^4_{L^2} -\overline{\lan A(\eta^T,\eta^T), \eta \ran}.
\end{equation}
\end{corollary}

\begin{proof}
Using (\ref{E:eta-T-evolution}) and Lemma \ref{L:DN}, noting that the last two terms of (\ref{E:eta-T-evolution}) are normal to $\Sigma$, we have
\[
\frac{1}{2} \frac{d}{dt} \|\eta^T\|^2_{L^2} = \overline{\lan \eta^T, \frac{\partial \eta^T}{\partial t} \ran} 
= - \frac{\ell}{2}  \| \DN \eta^T \|^2_{L^2}  + \frac{1}{\ell} \|\eta^T\|^4_{L^2} -\overline{\lan A(\eta^T,\eta^T), \eta^N \ran}.
\]
%This implies (\ref{E:eta-norm-evolution}) by noting that $\| \DN \eta^T \|^2_{L^2}= \frac{4}{\ell^2} \|\eta^T\|^2_{L^2}$.
\end{proof}

\begin{example}
\label{ex:flat2}
In the case of Example \ref{ex:flat}, we have $\eta^T(0)=-\eta^T(1)$ equals to a constant unit vector independent of $t$ and hence both sides are identically zero in (\ref{E:eta-T-evolution}) and (\ref{E:eta-norm-evolution}). 
\end{example}

\begin{example}
\label{ex:strip}
Consider the vertical strip $\Omega:=\{(x,y) \in  \mathbb{R}^2 : 0 \leq x \leq 1\}$ with boundary $\Sigma=\partial \Omega$ consists of two parallel vertical lines. Let $p_0=(0,-h/2)$ and $q_0=(1,h/2)$ for some $h>0$. It is easy to check that the solution to the chord shortening flow with initial chord $C_{p_0,q_0}$ is given by $p_t=(0,-h(t)/2)$, $q_t=(1,h(t)/2)$ where $h(t)$ is the unique solution to the ODE
\[ h'(t)=-\frac{2h(t)}{\sqrt{1+h^2(t)}} \]
with initial condition $h(0)=h$. From this we can see that the solution $h(t)$ exists for all $t \geq 0$. Moreover, $-h'(t) \leq 2h(t)$ implies $h(t) \leq h e^{-2t}$ and thus $h(t) \to 0$ exponentially as $t \to +\infty$. Therefore, the chord converges to a chord meeting $\partial \Omega$ orthogonally. In this case, we have
\[ -\eta^T(0)= \frac{1}{\sqrt{1+h^2(t)}} (0,h(t)) = \eta^T(1), \]
which satisfies the evolution equation (\ref{E:eta-T-evolution}) and $\eta^T \to 0$ as $t \to +\infty$. See Figure \ref{F:strip}.
\end{example}

\begin{figure}[h]
\centering
\includegraphics[height=6cm]{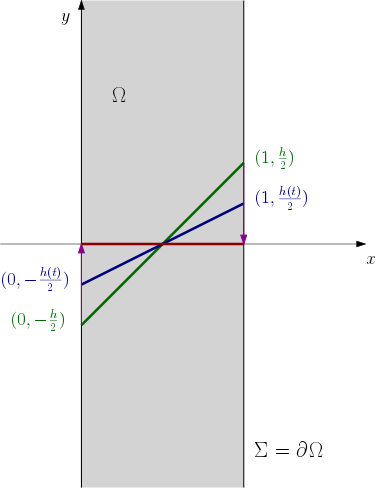}
\caption{A chord converging to a limit chord orthogonal to $\partial \Omega$.}
\label{F:strip}
\end{figure}

%%%%%%%%%%%%%%%%%%%%%%%%%%%%%%%%%%
% Section 		 Long time existence		     
\section{Long time existence}
\label{S:convergence}
%%%%%%%%%%%%%%%%%%%%%%%%%%%%%%%%%%

In this section, we prove our main convergence result which says that the only two possible convergence scenarios are given in Example \ref{ex:flat2} and \ref{ex:strip}. One should compare this convergence result with a similar result of Grayson \cite[Theorem 0.1]{Grayson89} for curve shortening flow on surfaces. For simplicity, we assume that $\Sigma$ is compact. However, the same result holds for non-compact $\Sigma$ which satisfies some convexity condition at infinity as in \cite{Grayson89}.

\begin{theorem}[Long time convergence]
\label{T:convergence}
Let $\Sigma \subset \mathbb{R}^n$ be a compact $k$-dimensional smooth submanifold without boundary. Suppose $C(0):I \to \mathbb{R}^n$ is a chord with distinct end points on $\Sigma$. Then there exists a maximally defined smooth family of chords $C(t):I \to \mathbb{R}^n$ for $t \in [0,T)$ with distinct end points on $\Sigma$, and $C(t)=C(\cdot,t)$ where $C(u,t):I \times [0,T) \to \R^n$ is the unique solution to the chord shortening flow (relative to $\Sigma$) as in Definition \ref{D:CSF}. 

Moreover, if $T < + \infty$, then $C(t)$ converges to a point on $\Sigma$ as $t \to T$. If $T$ is infinite, then $C(t)$ converges to an orthogonal geodesic chord with end points on $\Sigma$ as $t \to \infty$.
\end{theorem}

By the short time existence and uniqueness theorem (Theorem \ref{T:short-time-existence}), the chord shortening flow continues to exist and is unique as long as $\ell >0$. Therefore, $C(t)$ is uniquely defined for $t \in [0,T)$ where either $T < +\infty$ or $T=+\infty$. 

\begin{lemma}
\label{L:ell-limit}
Let $C(t)$, $t \in [0,T)$, be a maximally defined chord shortening flow. Then one of the following holds:
\begin{itemize}
\item[(a)] $T < +\infty$ and $C(t)$ shrinks to a point on $\Sigma$ as $t \to T$;
\item[(b)] $T=+\infty$ and $\ell(t) \to \ell_\infty >0$ as $t \to +\infty$.
\end{itemize}
\end{lemma}

\begin{proof}
As $\ell(t)$ is a non-increasing function of $t$ by Lemma \ref{L:length-evolution}, it either converges to $0$ or to some positive number $\ell_\infty>0$ as $t \to T$. By short time existence (Theorem \ref{T:short-time-existence}), it cannot converge to $\ell_\infty>0$ in finite time. So when $T<+\infty$, $C(t)$ must converge to a point on $\Sigma$ by compactness of $\Sigma$. It remains to show that $\ell(t)$ cannot converge to $0$ if $T=+\infty$. We will prove this by a contradiction argument. 
Suppose, on the contrary, that $T=+\infty$ and $\ell(t) \to 0$ as $t \to +\infty$. Since $\Sigma$ is compact, there exists some constant $\epsilon_0>0$ such that for any two points $p,q \in \Sigma$ with $d(p,q)<\epsilon_0$, the chord $C_{p,q}$ joining them has $\|\eta^T\|^2_{L^2}$ be bounded from below by a universal positive constant (see, for example, \cite[Lemma 5.2]{Colding-Minicozzi}). By Lemma \ref{L:length-evolution}, $\ell(t)$ must decrease to zero in finite time, which is a contradiction.
\end{proof}

Next, we claim that if the flow exists for all time (i.e. $T=+\infty$), then it must converge to an orthogonal geodesic chord to $\Sigma$ as $t \to \infty$. Since $|\eta^T| \leq \|\eta^T\|_{L^2}$, it suffices to prove the following lemma. Theorem \ref{T:convergence} clearly follows from Lemma \ref{L:ell-limit} and \ref{L:kto0}.

\begin{lemma}
\label{L:kto0}
Under the same assumption as Lemma \ref{L:ell-limit} and suppose $T=+\infty$, then $\|\eta^T\|_{L^2} \to 0$ as $t \to +\infty$.
\end{lemma}

\begin{proof}
Write $\ell_t=\ell(t)$ for $t \in [0,+\infty]$. By Lemma \ref{L:length-evolution} and \ref{L:ell-limit}, we have 
\begin{equation}
\label{E:ell-monotone}
\ell_0 \geq \ell_t \geq \ell_\infty>0 \quad \text{ for all $t$}. 
\end{equation}
Moreover, integrating the inequality in Lemma \ref{L:length-evolution} we obtain
\[ \ell_t- \ell_\infty = \int_{t}^\infty \|\eta^T\|^2_{L^2} \; d\tau \geq 0. \]
As a result, $\int_t^\infty \|\eta^T\|^2_{L^2} \; d\tau \to 0$ as $t \to \infty$. In other words, $\|\eta^T\|^2_{L^2}$ is $L^2$-integrable on $t \in [0,+\infty)$. If we can control the time derivative of $\|\eta^T\|_{L^2}^2$, then we can conclude that $\|\eta^T\|_{L^2} \to 0$ as $t \to \infty$. Using (\ref{E:eta-norm-evolution}), (\ref{E:ell-monotone}), Lemma \ref{L:DN} and $\|\eta^T \|^2_{L^2} \leq 2$, we have the following differential inequality
\begin{equation}
\label{E:norm-ineq} 
\frac{1}{2}\frac{d}{dt} \|\eta^T\|^2_{L^2} \leq \left( C+\frac{4}{\ell_\infty} \right) \|\eta^T\|^2_{L^2}
\end{equation}
where $C=\sup_{\Sigma} |A| > 0$ is a constant depending only on the compact submanifold $\Sigma$. We now combine (\ref{E:norm-ineq}) with the fact that $\int_t^\infty \|\eta^T\|^2_{L^2} \; d\tau \to 0$ as $t \to \infty$ to conclude that $\|\eta^T\|^2_{L^2} \to 0$ as $t \to \infty$.

To simplify notation, let $f(t):=\|\eta^T\|^2_{L^2}$ and $c:=C+\frac{4}{\ell_\infty}$. Then $\int_t^\infty f \to 0$ as $t \to \infty$ and $f' \leq cf$. We argue that $f(t) \to 0$ as $t \to \infty$. Suppose not, then there exists an increasing sequence $t_n \to +\infty$ such that 
\begin{equation}
\label{E:conv}
f(t_n) > \frac{1}{n} \qquad \text{and} \qquad \int_{t_n/2}^\infty f \leq \frac{1}{n^3}.
\end{equation}
We claim that there exists $t_n^* \in  (t_n -\frac{1}{n},t_n+\frac{1}{n})$ such that $f(t_n^*) \leq 1/n^2$. If not, then by (\ref{E:conv})
\[  \frac{2}{n^3} \leq \int_{t_n-\frac{1}{n}}^{t_n+\frac{1}{n}} f \leq \int_{t_n/2}^\infty f \leq \frac{1}{n^3}, \]
which is a contradiction. Using that $f' \leq c f$, we see that by (\ref{E:conv})
\[ \frac{1}{n} < f(t_n) \leq f(t_n^*) e^{\frac{c}{n}} \leq \frac{1}{n^2} e^{\frac{c}{n}}. \]
As a result, there is a contradiction when $n$ is sufficiently large. We have thus proved that $f(t) \to 0$ as $t \to \infty$, as claimed.
\end{proof}

%%%%%%%%%%%%%%%%%%%%%%%%%%%%%%%%%%%%%
% Section  	Existence of Orthogonal Geodesic Chords		      
\section{Existence of Orthogonal Geodesic Chords}
\label{S:applications}
%%%%%%%%%%%%%%%%%%%%%%%%%%%%%%%%%%%%%

In this section, we give several geometric applications of the chord shortening flow concerning the existence of multiple orthogonal geodesic chords. We first give the precise definition.

\begin{definition}
\label{D:OGC}
Let $\Sigma \subset \R^n$ be a smooth $k$-dimensional submanifold without boundary. An \emph{orthogonal geodesic chord for $\Sigma$} is a geodesic $c:[0,1] \to \R^n$ with endpoint $c(0)$ and $c(1)$ lying on $\Sigma$ such that $c'(0)$ and $c'(1)$ are normal to $\Sigma$ at $c(0)$ and $c(1)$ respectively.
\end{definition}

An \emph{orthogonal geodesic chord} is also called a \emph{free boundary geodesic} \cite{Zhou16} or a \emph{double normal} \cite{Kuiper64} in the literature. Note that in case $\Sigma \subset \R^n$ is an embedded hypersurface which bounds a domain $\Omega$ in $\R^n$, our definition of orthogonal geodesic chords does not require the chord to be contained inside $\overline{\Omega}$ as for example in \cite{Giambo-Giannoni-Piccione14}. The problem of the existence of multiple orthogonal geodesic chord for submanifolds in $\R^n$ was first treated by Riede \cite{Riede68} as follows. Let $\mathcal{C}_\Sigma$ be the space of all piecewise smooth curve $c:[0,1] \to \R^n$ with end points on $\Sigma$, endowed with the compact open topology. There exists a $\Z_2$-action on $\mathcal{C}_\Sigma$ by $c(t) \mapsto c(1-t)$ whose fixed point set is denoted by $\Delta'$. Denote by $H^{\Z_2}_*(\mathcal{C}_\Sigma,\Delta')$ and $H^*_{\Z_2}(\mathcal{C}_\Sigma)$ the $\Z_2$-equivariant homology groups (relative to $\Delta'$) and cohomology groups respectively. The following result is taken from \cite[Satz (5.5)]{Riede68}.

\begin{lemma}
\label{L:Riede}
If there exists $\beta \in H^{\Z_2}_*(\mathcal{C}_\Sigma,\Delta')$ and $\alpha_1,\cdots,\alpha_s \in H^*_{\Z_2}(\mathcal{C}_\Sigma)$ (not necessarily distinct) with deg $\alpha_i>0$ for all $i$ such that $(\alpha_1 \cup \cdots \cup \alpha_s) \cap \beta \neq 0$, then there exists at least $s+1$ orthogonal geodesic chords for $\Sigma$.
\end{lemma}  

The proof of Lemma \ref{L:Riede} involves a discrete curve shortening process $\Psi$ on $\mathcal{C}_\Sigma$ which satisfies properties (1) - (3) as described in the introduction. Since any curve $c \in \mathcal{C}_\Sigma$ can be continuously deformed into the unique chord joining the same end points, we can restrict $\mathcal{C}_\Sigma$ to the subset $\mathcal{C}^0_\Sigma$ consisting of all the chords with end points on $\Sigma$. The chord shortening flow is then a curve shortening process on $\mathcal{C}^0_\Sigma$ which satisfies all the required properties. Moreover, the space of chords $\mathcal{C}^0_\Sigma$ can also be described as the orbit space of $\Sigma \times \Sigma$ under the $\Z_2$-action $(p,q) \mapsto (q,p)$. As before, if we let $\Delta \subset \Sigma \times \Sigma$ be the fixed point set of the $\Z_2$-action, and $H^{\Z_2}_*(\Sigma \times \Sigma, \Delta)$, $H^*_{\Z_2}(\Sigma \times \Sigma)$ be the $\Z_2$-equivariant homology and cohomology respectively, we have by naturality
\begin{equation}
\label{E:isomorphism}
H^{\Z_2}_*(\Sigma \times \Sigma,\Delta) \cong H^{\Z_2}_*(\mathcal{C}_\Sigma,\Delta') \qquad \text{and} \qquad H^*_{\Z_2}(\Sigma \times \Sigma) \cong H^*_{\Z_2}(\mathcal{C}_\Sigma).
\end{equation}
In \cite{Hayashi82}, Hayashi studied the equivariant (co)homology of $\Sigma \times \Sigma$ and obtained the following result in \cite[Theorem 2]{Hayashi82}.

\begin{lemma}
\label{L:Hayashi}
There exists $\beta \in H^{\Z_2}_{2k}(\Sigma \times \Sigma,\Delta)$ and $\alpha \in H^1_{\Z_2}(\Sigma \times \Sigma)$ such that $\alpha^k \cap \beta \neq 0$ in $H^{\Z_2}_k(\Sigma \times \Sigma, \Delta)$, where $\alpha^k=\alpha \cup \cdots \cup \alpha$ is the $k$-th power of cup products of $\alpha$ and $k=\dim \Sigma$.
\end{lemma}

Combining Lemma \ref{L:Hayashi}, \ref{L:Riede} and (\ref{E:isomorphism}), we have proved Theorem \ref{T:double-normals}, which clearly implies Lusternik-Schnirelmann's theorem (Theorem \ref{T:LS}) as a special case since the orthogonal geodesic chords must be contained inside the convex domain by convexity of the domain $\Omega \subset \R^n$.

%%%%%%%%%%%%%%%%%%%%%%%%%%%%%%%%%%%%%
% Section  			Shrinking convex chord to a point    				      
\section{Shrinking convex chord to a point}
\label{S:convex-chord}
%%%%%%%%%%%%%%%%%%%%%%%%%%%%%%%%%%%%%

In this section, we study the evolution of chords inside a convex connected planar domain in $\R^2$. In particular, we prove that if an initial chord is \emph{convex}, then it will shrink to a point in finite time under the chord shortening flow. In order to make precise the concept of \emph{convexity}, we need to be consistent with the orientation of a curve in $\R^2$. For this reason, we restrict our attention to plane curves which bounds a domain in $\R^2$.

\begin{definition}[Boundary orientation]
\label{D:boundary-orientation}
For any smooth domain $\Omega \subset \mathbb{R}^2$, we always orient the boundary $\partial \Omega$ as the boundary of $\Omega$ with the standard orientation inherited from $\mathbb{R}^2$. The orientation determines uniquely a global unit tangent vector field, called the \emph{orientation field}, $\xi:\partial \Omega \to T(\partial \Omega)$ such that $\nu:=J \xi$ is the inward pointing normal of $\partial \Om$ relative to $\Om$. Here, $J:\R^2 \to \R^2$ is the counterclockwise rotation by $\pi/2$.
\end{definition}

Using Definition \ref{D:boundary-orientation}, we can define the \emph{boundary angle} $\Theta$ which measures the contact angle between a chord $C$ and the boundary $\partial \Omega$. 

\begin{definition}[Boundary angle]
\label{D:Theta}
For any (oriented) chord $C_{p,q}$ joining $p$ to $q$ with $p \neq q \in \partial \Omega$, we define the \emph{boundary angle} $\Theta:\{p,q\} \to \mathbb{R}$ by
\[ \Theta(p):=\lan \eta(p), \xi(p) \ran, \quad \text{and} \quad \Theta(q):=- \lan \eta(q), \xi(q) \ran,\]
where $\xi$ is the orientation field on $\partial \Omega$ as in Definition \ref{D:boundary-orientation}.
\end{definition}

\begin{definition}
\label{D:convex-chord}
An oriented chord $C_{p,q}$ is \emph{convex} if $\Theta \geq 0$ at both end points.
\end{definition}

\begin{remark}
If we change the orientation of the chord from $C_{p,q}$ to $C_{q,p}$, the boundary angle $\Theta$ changes sign. Since the orientation field $\xi$ is always tangent to $\partial \Omega$, we have $\Theta(p)=\Theta(q)=0$ if and only if $C_{p,q}$ meets $\partial \Omega$ orthogonally at its end points $p$ and $q$.
\end{remark}

If we define the ``unit normal'' $N$ of $\partial C_{p,q}=\{p,q\}$ inside $\partial \Omega$ by setting
\[  N(p)=-\xi(p) \quad \text{and} \quad N(q)=\xi(q),\]
then a solution to the chord shortening flow (\ref{E:CSF'}) can be consider as a smooth $1$-parameter family of pair of points on $\partial \Om$ given by $\gamma:\{0,1\} \times [0,T) \to \partial \Omega$ such that
\begin{equation}
\label{E:CSF''}
\frac{\partial \gamma}{dt} (u,t) =\Theta (\gamma(u,t)) N(\gamma(u,t)),
\end{equation}
where $\Theta$ is the boundary angle for the oriented chord from $\gamma(0,t)$ to $\gamma(1,t)$. Since the value of $\Theta$ at $u=0$ depends also on the other end point $\gamma(1,t)$, this is a non-local function. Therefore, the chord shortening flow can be thought of as a non-local curve shortening flow driven by the boundary angle $\Theta$.

We are now ready to state the main theorem of this section. The readers can compare Theorem \ref{T:contraction} with the famous result of Huisken \cite{Huisken84} which says that any compact embedded convex hypersurface in $\R^n$ would contract to a point in finite time under the mean curvature flow.

\begin{theorem}
\label{T:contraction}
Let $\Omega \subset \R^2$ be a compact connected domain with smooth convex boundary. Any convex chord which is not an orthogonal geodesic chord would converge to a point in finite time under the chord shortening flow.
\end{theorem}

To prove Theorem \ref{T:contraction} we need to establish a few propositions, which are of geometric interest. We first state the evolution of the boundary angle $\Theta$ under the chord shortening flow. Note that we always have $|\Theta| \leq 1$ by definition.

\begin{proposition} [Evolution of boundary angle]
\label{P:Theta-evolution}
Suppose $C(u,t):I \times [0,T) \to \R^2$ is a solution to the chord shortening flow as in Definition \ref{D:CSF}. Then, the boundary angle $\Theta(u,t): \{0,1\} \times [0,T) \to \R$ satisfies the following equation (recall (\ref{E:norm}) and (\ref{E:DN}))
%Let $C(t)=C_{p_t,q_t}$, where $t \in [0,T)$, be a solution to the chord shortening flow. Then we have
\begin{equation}
\label{E:Theta-evolution}
\frac{\partial }{\partial t} \Theta = -\Delta^{\frac{1}{2}} \Theta+  \frac{1}{\ell} \Big(\|\Theta\|^2_{L^2}+ \ell k \langle -\eta,\nu \rangle  \Big) \Theta + \frac{1}{\ell} \Big( 1+ \langle \xi(p),\xi(q)\rangle \Big) (\Theta-\overline{\Theta}),
\end{equation}
where $k:=\langle \nabla_\xi \xi,\nu \rangle$ is the curvature of $\partial \Omega$ with respect to $\nu$ (recall Definition \ref{D:boundary-orientation}), $\ell=\ell(t)$ is the length of the chord $C(\cdot,t):I \to \R^2$ with outward unit conormal $\eta$. 
%Here, $\DN$ is the Dirichlet-to-Neumann operator for the chord $C(\cdot,t):I \to \R^2$ as in Definition \ref{D:DN-map}.
\end{proposition}

\begin{proof}
It follow directly from Theorem \ref{T:eta-T-evolution} and Definition \ref{D:Theta}
\end{proof}

Using (\ref{E:Theta-evolution}), we immediately have the following evolution equations.

\begin{corollary}
Under the same hypothesis as Proposition \ref{P:Theta-evolution}, we have:
\begin{eqnarray}
\label{E:Theta-avg-evolution}
\frac{d}{dt} \overline{\Theta} &=& \frac{1}{\ell} \Big( \|\Theta\|^2_{L^2} -1-\langle \xi(p),\xi(q) \rangle \Big) \overline{\Theta} + \overline{k \langle -\eta, \nu \rangle \Theta}, \\
\label{E:Theta-norm-evolution}
\frac{1}{2} \frac{d}{dt} \|\Theta\|^2_{L^2} &=&  \frac{\ell}{2} \lan \xi(p),\xi(q) \ran \|\DN \Theta\|^2_{L^2} + \overline{k \langle -\eta,\nu \rangle \Theta^2} \\
&& \hspace{2cm} + \frac{1}{\ell}  \Big( \|\Theta\|^2_{L^2}-1-\langle \xi(p),\xi(q) \rangle \Big) \|\Theta\|^2_{L^2}. \nonumber
\end{eqnarray}
\end{corollary}

\begin{proof}
Both equation follows from (\ref{E:Theta-evolution}) and Lemma \ref{L:DN}.
\end{proof}

%\begin{proof}
%(\ref{E:Theta-avg-evolution}) and (\ref{E:Theta-norm-evolution}) follow directly from (\ref{E:Theta-evolution}) and the fact that $\overline{\DN f}=0$ for any $f:\partial C_{p,q}=\{p,q\} \to \R$.
%\end{proof}

Our first lemma is that convexity is preserved under the chord shortening flow. From now on, we will use $C(t)$ to denote the unique solution to the chord shortening flow with initial chord $C(0)$ defined on the maximal time interval $t \in [0,T)$ (where $T$ could be infinite). 

\begin{lemma}
\label{L:convex}
Let $C(0)$ be a convex chord inside a compact domain $\Omega \subset \R^2$ with convex boundary $\partial \Omega$. Then, $C(t)$ remains convex for all $t \in [0,T)$.
\end{lemma}

\begin{proof}
Let $\Theta_{min}$ and $\Theta_{max}$ be the minimum and maximum of $\Theta$,  both of which is a Lipschitz function of $t$. By (\ref{E:Theta-evolution}), we have the following equality
\begin{equation}
\label{E:Theta-min-1}
\frac{d}{dt} \Theta_{min} = \frac{1}{\ell} \Big(  (\|\Theta\|^2_{L^2}-1) \Theta_{min} + \ell k \langle -\eta,\nu \rangle \Theta_{min} -  \langle \xi(p),\xi(q)\rangle \Theta_{max}  \Big)
\end{equation}
As $\partial \Omega$ is convex, we have $k \geq 0$ and $\langle -\eta,\nu \rangle \geq 0$. Moreover, if the chord is convex, then $\Theta_{min} \geq 0$. Therefore, (\ref{E:Theta-min-1}) implies the following differential inequality
\begin{equation}
\label{E:Theta-min-2}
 \frac{d}{dt} \Theta_{min} \geq \frac{1}{\ell} \Big( (\|\Theta\|^2_{L^2}-1) \Theta_{min}  -  \langle \xi(p),\xi(q)\rangle \Theta_{max}  \Big).
\end{equation}
By elementary geometry (see Figure \ref{F:xi}), we can express the term involving the orientation field as
\begin{equation}
\label{E:xi}
\langle \xi(p),\xi(q) \rangle = \Theta_p \Theta_q - \sqrt{(1-\Theta_p^2)(1-\Theta_q^2)}. 
\end{equation}

Combining (\ref{E:Theta-min-2}) with (\ref{E:ell-evolution}), noting that $\|\eta^T\|^2_{L^2}=\|\Theta\|^2_{L^2}$ and using (\ref{E:xi}),
\begin{eqnarray*}
 \frac{d}{dt} \left( \frac{\Theta_{min}}{\ell} \right) &\geq& \frac{1}{\ell^2} \Big( (2\|\Theta\|^2_{L^2}-1) \Theta_{min}  -  \langle \xi(p),\xi(q)\rangle \Theta_{max} \Big) \\
 &=& \frac{1}{\ell^2} \Big( 2\Theta^3_{min} -(1-\Theta_{max}^2)\Theta_{min} + \sqrt{(1-\Theta_{min}^2)(1-\Theta_{max}^2)} \Theta_{max}   \Big) \\
 & \geq &  \frac{1}{\ell^2} \Big( 2 \Theta_{min}^3 + (1-\Theta_{max}^2) (\Theta_{max} -\Theta_{min}) \Big) \geq 0.
\end{eqnarray*}
Therefore, if $\Theta_{min} \geq 0$ at $t=0$, then $\Theta_{min} /\ell$ is a non-decreasing function of $t$, hence is non-negative for all $t \in [0,T)$. This proves that $C(t)$ remains convex for all $t \in [0,T)$.
\end{proof}

\begin{figure}[h]
\centering
\includegraphics[height=5cm]{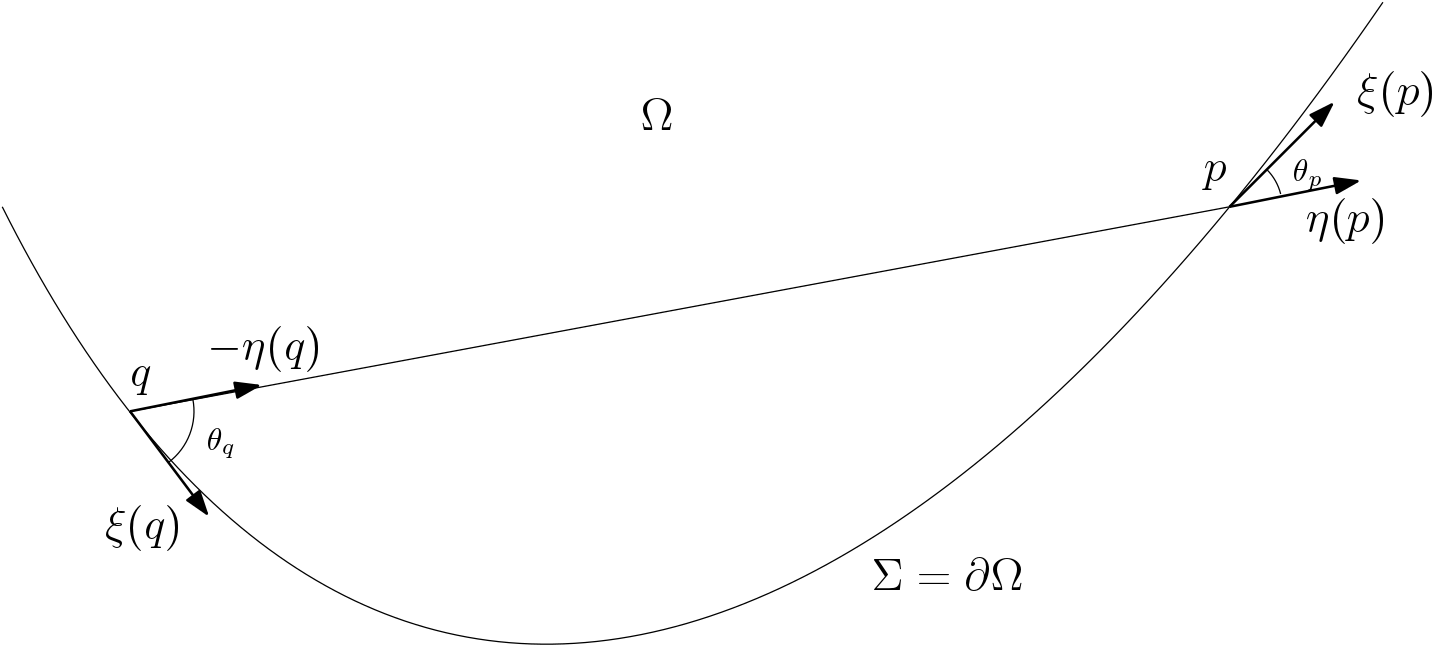}
\caption{The convex region cut out by a convex chord in $\Omega$. Note that $\langle \xi(p) ,\xi(q) \rangle =\cos (\theta_p+\theta_q)$.}
\label{F:xi}
\end{figure}

We are now ready to prove the main result of this section.

\begin{proof}[Proof of Theorem \ref{T:contraction}]
By Theorem \ref{T:convergence}, it suffices to show that the chord shortening flow $C(t)$ exists only on a maximal time interval $t \in [0,T)$ with $T<+\infty$. First of all, $\Theta \geq 0$ for all $t \in [0,T)$ by Lemma \ref{L:convex}. Using (\ref{E:Theta-avg-evolution}) and (\ref{E:ell-evolution}), notice that $2\|\Theta\|^2_{L^2} \geq \overline{\Theta}^2$, a similar argument as in the proof of Lemma \ref{L:convex} gives\begin{eqnarray*}
\frac{d}{dt} \left( \frac{\overline{\Theta}}{\ell} \right) &\geq & \frac{1}{\ell^2} \Big(\overline{\Theta}^2-1-\langle \xi, \xi \rangle \Big) \overline{\Theta} \\
& \geq & \frac{1}{\ell^2} \Big( \Theta_{min}^2 + \Theta_{min} \Theta_{max} \Big) \overline{\Theta} \geq 0.
\end{eqnarray*}
Therefore, $\overline{\Theta}/\ell$ is a non-decreasing function of $t$. Since $\overline{\Theta}/\ell >0$ at $t=0$, it remains bounded away from zero for all $t \in [0,T)$. Therefore, if $T=+\infty$, by Theorem \ref{T:convergence} we must have $C(t)$ converges to an orthogonal geodesic chord and thus $\overline{\Theta}/\ell \to 0$, which is a contradiction.
\end{proof}

%Finally, we end our paper with the following question. \textit{What happens if the convexity assumption of $\partial \Omega$ is dropped in Theorem \ref{T:contraction}?} A famous result of Huisken \cite{Huisken86} says that any ``sufficiently convex'' embedded hypersurface in a Riemannian manifold would converge to a point in finite time under the mean curvature flow. It is tempting to think that a chord would still converge to a point in finite time if the initial chord is ``convex enough'' to overcome the curvature (and its derivatives) of $\partial \Omega$. However, if we take $\Omega$ to be the complement of two round disks of the same radius in $\R^2$, there exists a chord which is lying inside $\Omega$ and is \emph{tangential} to $\partial \Omega$ at its end points, hence $\Theta \equiv 1$. This chord is as convex as possible but does not converge to a point in finite time. However, if we further impose that the initial chord is short (relative to the curvature bound of $\partial \Omega$), will it converge to a point in finite time?   

\bibliographystyle{amsplain}
\bibliography{references}

\end{document}